\documentclass{article}

\usepackage{
	amsmath, 
	amsthm, 
	amssymb,
	fancyhdr, 
	setspace,
	tikz-cd,
}
\usepackage{upgreek}

\usepackage{lscape}
\usepackage{color}
\usepackage[color,matrix,arrow]{xy}
\usepackage{setspace}
\xyoption{all}

\DeclareMathOperator{\Ext}{Ext}
\DeclareMathOperator{\Hom}{Hom}
\DeclareMathOperator{\pd}{pd}
\DeclareMathOperator{\id}{id}
\DeclareMathOperator{\add}{add}

\DeclareMathOperator{\End}{End}

\usepackage{avant}
\usepackage{amsmath, amssymb, amsfonts, color, tikz, bbm, txfonts, euscript}
\usepackage{amssymb}
\usepackage{enumerate}
\usepackage{bm}
\newtheorem{theorem}{Theorem}[section]
\newtheorem{lemma}[theorem]{Lemma}

\newtheorem{prop}[theorem]{Proposition}

\theoremstyle{definition}
\newtheorem{mydef}[theorem]{Definition}

\begin{document}

\thispagestyle{empty}

\title{1-Auslander-Gorenstein Algebras Which Are Tilted}
\author{Stephen Zito}
\maketitle

\begin{abstract}
Let $\Lambda$ be a $1$-Auslander-Gorenstein algebra.  We give a necessary and sufficient condition for $\Lambda$ to be a tilted algebra.
\end{abstract}

\section{Introduction}
We set the notations for the remainder of this paper. All algebras are assumed to be finite dimensional over an algebraically closed field $k$.  If $\Lambda$ is a $k$-algebra then denote by $\mathop{\text{mod}}\Lambda$ the category of finitely generated right $\Lambda$-modules and by $\mathop{\text{ind}}\Lambda$ a set of representatives of each isomorphism class of indecomposable right $\Lambda$-modules.  Given $M\in\mathop{\text{mod}}\Lambda$, the projective dimension of $M$ in $\mathop{\text{mod}}\Lambda$ is denoted by $\pd_{\Lambda}M$ and its injective dimension by $\id_{\Lambda}M$.  We denote by $\add M$ the smallest additive full subcategory of $\mathop{\text{mod}}\Lambda$ containing $M$, that is, the full subcategory of $\mathop{\text{mod}}\Lambda$ whose objects are the direct sums of direct summands of the module $M$.  We let $\tau_{\Lambda}$ and $\tau^{-1}_{\Lambda}$ be the Auslander-Reiten translations in $\mathop{\text{mod}}\Lambda$.  $D$ will denote the standard duality functor $\Hom_k(-,k)$.  Given a $\Lambda$-module $M$, denote the kernel of a projective cover by $\Omega_{\Lambda} M$.  $\Omega_{\Lambda} M$ is called the syzygy of $M$.  Denote the cokernel of an injective envelope by $\Omega_{\Lambda}^{-1}M$.  $\Omega_{\Lambda}^{-1}M$ is called the cosyzygy of $M$.  Finally,  let $\mathop{\text{gl.dim}}\Lambda$ stand for the global dimension and $\mathop{\text{domdim}}\Lambda$ stand for the dominant dimension of an algebra $\Lambda$ (see Definition $\ref{def3}$).

Let $\Lambda$ be an algebra of finite type.  A $\Lambda$-module $M$ is an additive generator of $\mathop{\text{mod}}\Lambda$ provided every indecomposable $\Lambda$-module is a direct summand of $M$.   Auslander in $\cite{AU}$ characterized the endomorphism algebras of additive generators as those algebras which have global dimension at most $2$ and dominant dimension at least $2$.  The algebras occurring in this way are now called Auslander algebras.
\par
Another characterization of Auslander algebras concerns properties of a certain subcategory of $\mathop{\text{mod}}\Lambda$.  Let $\tilde{Q}$ be the direct sum of representatives of the isomorphism classes of all indecomposable projective-injective $\Lambda$-modules.  Let $\mathcal{C}_{\Lambda}$ be the full subcategory consisting of all modules generated and cogenerated by $\tilde{Q}$ (see Definition $\ref{def1}$ and Definition $\ref{def2}$).  When $\mathop{\text{gl.dim}}\Lambda=2$, Crawley-Boevey and Sauter showed in $\cite{CBS}$ that the algebra $\Lambda$ is an Auslander algebra if and only if there exists a tilting-cotilting $\Lambda$-module $T_{\mathcal{C}}$ in $\mathcal{C}_{\Lambda}$ (see Definition $\ref{Tilting}$).
\par
Recent work by Nguyen, Reiten, Todorov, and Zhu in $\cite{NRTZ}$ has showed the existence of a tilting or cotilting module in $\mathcal{C}_{\Lambda}$ is equivalent to the dominant dimension of $\Lambda$ being at least $2$ without any condition on the global dimension of $\Lambda$.
Iyama and Solberg defined $m$-Auslander-Gorenstein algebras in $\cite{OS}$ as algebras where the dominant dimension is equal to the Gorenstein dimension. (see Definition $\ref{def4}$).  In particular, a $1$-Auslander-Gorenstein algebra $\Lambda$ is either a selfinjective algebra or a Goresntein algebra satisfying $\id_{\Lambda}\Lambda_{\Lambda}=\mathop{\text{domdim}}\Lambda=2$.  Nguyen, Reiten, Todorov, and Zhu also showed that an algebra $\Lambda$ is $1$-Auslander-Gorenstein if and only if $\mathcal{C}_{\Lambda}$ contains a tilting-cotilting module.  Their results suggest $1$-Auslander-Gorenstein algebras are generalizations of Auslander algebras.
\par
Tilting theory is one of the main themes in the study of the representation theory of algebras.  Tilted algebras were introduced by Happel and Ringel in $\cite{HR}$.  Tilted algebras have been the subject of much investigation and several characterizations are known.  An interesting question concerns when an Auslander algebra is also a tilted algebra.  This problem was studied and answered by Zito in $\cite{Z}$.
The proof relied heavily on a recent characterization of tilted algebras by Jaworska, Malicki, and Skowro$\acute{\text{n}}$ski in $\cite{JMS}$.  They showed an algebra $\Lambda$ is tilted if and only if there exists a sincere module $M$ such that $\text{Hom}_{\Lambda}(M,\tau_{\Lambda}X)=0$ or $\text{Hom}_{\Lambda}(X,M)=0$ for every module $X\in\mathop{\text{ind}}\Lambda$.  Zito was able to show the tilting-cotilting module in $\mathcal{C}_{\Lambda}$ satisfied these conditions.
\par  
Considering $1$-Auslander-Gorenstein algebras as generalizations of Auslander algebras, it's natural to ask when such an algebra is also a tilted algebra. In this paper, we give a necessary and sufficient condition for a $1$-Auslander-Gorenstein algebra $\Lambda$ to be tilted.  The proof involves studying a particular full subcategory induced by $T_{\mathcal{C}}$ and taking advantage of certain homological properties of tilted algebras.

Here, we let  $T_{\mathcal{C}}$ be the tilting-cotilting module in $\mathcal{C}_{\Lambda}$, $\text{Cogen}T_{\mathcal{C}}$ be the class of modules $X$ in $\mathop{\text{mod}}\Lambda$ cogenerated  by $T_{\mathcal{C}}$, $\mathcal{P}^1(\Lambda)$ be the full subcategory of $\mathop{\text{mod}}\Lambda$ whose objects are all the $\Lambda$-modules $Y$ such that $\pd_{\Lambda}Y\leq1$, and $\mathcal{L}_{\Lambda}$ be the full subcategory of $\mathop{\text{mod}}\Lambda$ consisting of all indecomposable $\Lambda$-modules whose predecessors have projective dimension at most $1$ (see Definition $\ref{def6}$).
\begin{theorem}$\emph{(Theorem~2.1)}.$
Suppose $\Lambda$ is a $1$-Auslander-Gorenstein algebra.  Then $\Lambda$ is a tilted algebra if and only if $\emph{add}\mathcal{L}_{\Lambda}=\emph{Cogen}T_{\mathcal{C}}$.
\end{theorem}

\subsection{Properties of the Subcategory $\mathcal{C}_{\Lambda}$}
Let $\Lambda$ be an algebra.
\begin{mydef}
\label{def1}
Let $M$ be a $\Lambda$-module.  We define $\mathop{Gen} M$ to be the class of all modules $X$ in $\mathop{\text{mod}}\Lambda$ generated by $M$, that is, the modules $X$ such that there exists an integer $d\geq0$ and an epimorphism $M^d\rightarrow X$ of $\Lambda$-modules.  Here, $M^d$ is the direct sum of $d$ copies of $M$.  Dually, we define $\mathop{Cogen}M$ to be the class of all modules $Y$ in $\mathop{\text{mod}}\Lambda$ cogenerated  by $M$, that is, the modules $Y$ such that there exist an integer $d\geq0$ and a monomorphism $Y\rightarrow M^d$ of $\Lambda$-modules.
\end{mydef}
\begin{mydef}
\label{def2}
Let $\tilde{Q}$ be the direct sum of representatives of the isomorphism classes of all indecomposable projective-injective $\Lambda$-modules.  Let $\mathcal{C}_{\Lambda}:=(\text{Gen}\tilde{Q})\cap(\text{Cogen}\tilde{Q})$ be the full subcategory consisting of all modules generated and cogenerated by $\tilde{Q}$.
\end{mydef}
When $\mathop{\text{gl.dim}}\Lambda=2$, Crawley-Boevey and Sauter showed the following characterization of Auslander algebras.
\begin{lemma}$\emph{\cite[Lemma~1.1]{CBS}}$
\label{CBS}
If $\mathop{\emph{gl.dim}}\Lambda=2$, then $\mathcal{C}_{\Lambda}$ contains a tilting-cotilting module if and only if $\Lambda$ is an Auslander algebra.
\end{lemma}
Nguyen, Reiten, Todorov, and Zhu in $\cite{NRTZ}$ studied the existence of such tilting and cotilting modules without any condition on the global dimension of $\Lambda$ and gave a precise description.  We first recall the definition of the dominant dimension of an algebra.
\begin{mydef}
\label{def3}
Let $\Lambda$ be an algebra and let
\begin{center}
$0\rightarrow\Lambda_{\Lambda}\rightarrow I_0\rightarrow I_1\rightarrow\ I_2\rightarrow\cdots$
\end{center}
be a minimal injective resolution of $\Lambda$.  Then $\mathop{\text{domdim}}\Lambda=n$ if $I_i$ is projective for $0\leq i\leq n-1$ and $I_n$ is not projective.  If all $I_n$ are projective, we say $\text{domdim}\Lambda=\infty$.
\end{mydef}

\begin{theorem}$\emph{\cite[Corollary~2.4.2]{NRTZ}}$
\label{NRTZ Main}
Let $\mathcal{C}_{\Lambda}$ be the full subcategory consisting of all modules generated and cogenerated by $\tilde{Q}$
\begin{enumerate}
\item[\emph{(1)}] The following statements are equivalent:
\begin{enumerate}
\item[\emph{(a)}] $\mathop{\emph{domdim}}\Lambda\geq2$
\item[\emph{(b)}] $\mathcal{C}_{\Lambda}$ contains a tilting module $T_{\mathcal{C}}$
\item[\emph{(c)}] $\mathcal{C}_{\Lambda}$ contains a cotilting module $C_{\mathcal{C}}$.
\end{enumerate}
\item[\emph{(2)}] If a tilting module $T_{\mathcal{C}}$ exists in $\mathcal{C}_{\Lambda}$, then $T_{\mathcal{C}}\cong\tilde{Q}\oplus(\bigoplus_i\Omega_{\Lambda}^{-1}P_i)$, where $\Omega_{\Lambda}^{-1}P_i$ is the cosyzygy of $P_i$ and the direct sum is taken over representatives of the isomorphism classes of all indecomposable projective non-injective $\Lambda$-modules $P_i$.
\item[\emph{(3)}] If a cotilting module $C_{\mathcal{C}}$ exists in $\mathcal{C}_{\Lambda}$, then $C_{\mathcal{C}}\cong\tilde{Q}\oplus(\bigoplus_i\Omega_{\Lambda}I_i)$ where $\Omega_{\Lambda}I_i$ is the syzygy of $I_i$ and the direct sum is taken over representatives of the isomorphism classes of all indecomposable injective non-projective $\Lambda$-modules $I_i$.
\end{enumerate}
\end{theorem}

Recently, Iyama and Solberg defined $m$-Auslander-Gorenstein algebras in $\cite{OS}$.  They also showed that the notion of $m$-Auslander-Gorenstein algebra is left and right symmetric.
\begin{mydef}
\label{def4}
An algebra $\Lambda$ is $m$-$\emph{Auslander-Gorenstein}$ if
\begin{center}
$\id_{\Lambda}\Lambda_{\Lambda}\leq m+1\leq\mathop{\text{domdim}}\Lambda$.
\end{center}
\end{mydef}
In particular, a $1$-Auslander-Gorenstein algebra $\Lambda$ is either a selfinjective algebra or a Goresntein algebra satisfying $\id_{\Lambda}\Lambda_{\Lambda}=\mathop{\text{domdim}}\Lambda=2$.  Nguyen, Reiten, Todorov, and Zhu provide a characterization of $1$-Auslander-Gorenstein algebras in terms of the existence of a tilting-cotilting module in $\mathcal{C}_{\Lambda}$.
\begin{theorem}$\emph{\cite[Theorem~2.4.12]{NRTZ}}$
Let $\Lambda$ be an algebra.  Then the subcategory $\mathcal{C}_{\Lambda}$ contains a tilting-cotilting module if and only if $\Lambda$ is a $1$-Auslander-Gorenstein algebra.
\end{theorem}
The following statement generalizes Crawley-Boevey and Sauter's result from $\cite{CBS}$.
\begin{prop}$\emph{\cite[Corollary~2.4.13]{NRTZ}}$
\label{general}
Let $\Lambda$ be an algebra with $\mathop{\emph{gl.dim}}\Lambda<\infty$.  Then the subcategory $\mathcal{C}_{\Lambda}$ contains a tilting-cotilting module if and only if $\Lambda$ is an Auslander algebra.
\end{prop}

We will need the following lemma.
\begin{lemma}$\emph{\cite[Lemma~1.1.4]{NRTZ}}$
\label{small}
Let $\Lambda$ be an algebra.  Let $X\in\mathcal{C}_{\Lambda}$.  Let $Y$ be a $\Lambda$-module with $\pd_{\Lambda}Y=1$.  Then $\emph{Ext}_{\Lambda}^1(Y,X)=0$.
\end{lemma}

\subsection{Tilting and Cotilting Modules}
  We begin with the definition of tilting and cotilting modules.
   \begin{mydef}
   \label{Tilting}
    Let $\Lambda$ be an algebra.  A $\Lambda$-module $T$ is a $\emph{partial tilting module}$ if the following two conditions are satisfied: 
   \begin{enumerate}
   \item[($\text{1}$)] $\pd_{\Lambda}T\leq1$.
   \item[($\text{2}$)] $\Ext_{\Lambda}^1(T,T)=0$.
   \end{enumerate}
   A partial tilting module $T$ is called a $\emph{tilting module}$ if it also satisfies:
   \begin{enumerate}
   \item[($\text{3}$)] There exists a short exact sequence $0\rightarrow \Lambda\rightarrow T'\rightarrow T''\rightarrow 0$ in $\mathop{\text{mod}}\Lambda$ with $T'$ and $T''$ $\in \add T$.
   \end{enumerate}
   A $\Lambda$-module $C$ is a $\emph{partial cotilting module}$ if the following two conditions are satisfied: 
   \begin{enumerate}
   \item[($\text{1}'$)] $\id_{\Lambda}C\leq1$.
   \item[($\text{2}'$)] $\Ext_{\Lambda}^1(C,C)=0$.
   \end{enumerate}   
   A partial cotilting module is called a $\emph{cotilting module}$ if it also satisfies:
   \begin{enumerate}
   \item[($\text{3}'$)] There exists a short exact sequence $0\rightarrow C'\rightarrow C''\rightarrow D\Lambda\rightarrow 0$ in $\mathop{\text{mod}}\Lambda$ with $C'$ and $C''$ $\in \add C$.   
  \end{enumerate} 
   \end{mydef}
 
 Tilting modules and cotilting modules induce torsion pairs in a natural way.  We consider the restriction to a subcategory $\mathcal{C}$ of a functor $F$ defined originally on a module category, and we denote it by $F|_{\mathcal{C}}$.   
 \begin{mydef} 
 \label{def5}
 A pair of full subcategories $(\mathcal{T},\mathcal{F})$ of $\mathop{\text{mod}}\Lambda$ is called a $\emph{torsion pair}$ if the following conditions are satisfied:
   \begin{enumerate}
   \item[(1)] $\text{Hom}_{\Lambda}(M,N)=0$ for all $M\in\mathcal{T}$, $N\in\mathcal{F}.$
   \item[(2)] $\text{Hom}_{\Lambda}(M,-)|_\mathcal{F}=0$ implies $M\in\mathcal{T}.$
   \item[(3)] $\text{Hom}_{\Lambda}(-,N)|_\mathcal{T}=0$ implies $N\in\mathcal{F}.$
   \end{enumerate}
   \end{mydef}
   Consider the following full subcategories of $\mathop{\text{mod}}\Lambda$ where $T$ is a tilting $\Lambda$-module.
 \[
 \mathcal{T}(T)=\{M\in\mathop{\text{mod}}\Lambda~|~ \text{Ext}_{\Lambda}^{1}(T,M)=0\}
 \]
 \[
 \mathcal{F}(T)=\{M\in\mathop{\text{mod}}\Lambda~|~\text{Hom}_{\Lambda}(T,M)=0\}
 \]
 Then $(\text{Gen}T,\mathcal{F}(T))=(\mathcal{T}(T),\text{Cogen}(\tau_{\Lambda} T))$ is a torsion pair in $\mathop{\text{mod}}\Lambda$. 
 Consider the following full subcategories of $\mathop{\text{mod}}\Lambda$ where $C$ is a cotilting $\Lambda$-module. 
  \[
 \mathcal{T}(C)=\{M\in\mathop{\text{mod}}\Lambda~|~ \text{Ext}_{\Lambda}^{1}(M,C)=0\}
 \]
 \[
 \mathcal{F}(C)=\{M\in\mathop{\text{mod}}\Lambda~|~\text{Hom}_{\Lambda}(M,C)=0\}
 \] 
 Then $(\text{Gen}(\tau_{\Lambda}^{-1}C),\mathcal{T}(C))=(\mathcal{F}(C),\text{Cogen}(C))$ is a torsion pair in $\mathop{\text{mod}}\Lambda$.  We refer the reader to $\cite{ASS}$ for more details.

 \par
 \begin{mydef}
Let $\mathcal{T}$ be a full subcategory of $\mathop{\text{mod}}\Lambda$.  We say a $\Lambda$-module $X\in\mathcal{T}$ is $\text{Ext}$-$\it{projective}$ if $\text{Ext}_{\Lambda}^1(X,-)|_{\mathcal{T}}=0$.
\end{mydef} 
The following proposition characterizes Ext-projectives in torsion classes.
\begin{prop}$\emph{\cite[VI.1,~Proposition~1.11]{ASS}}$
 \label{Ext}
 Let $(\mathcal{T},\mathcal{F})$ be a torsion pair in $\mathop{\emph{mod}}\Lambda$ and $X\in\mathcal{T}$ be an indecomposable $\Lambda$-module.  Then $X$ is $\emph{Ext}$-projective in $\mathcal{T}$ if and only if $\tau_{\Lambda}X\in\mathcal{F}$.
 \end{prop}
 We need the following characterization of tilting modules.
 \begin{theorem}$\emph{\cite[VI.2,~Theorem~2.5]{ASS}}$ 
 \label{Tilting}
 Let $T$ be a partial tilting $\Lambda$-module.  Then $T$ is a tilting $\Lambda$-module if and only if, for every $X\in\mathop{\emph{mod}}\Lambda$, $X\in\add T$ if and only if $X$ is $\emph{Ext}$-projective in $\mathcal{T}(T)$.
 \end{theorem}

 We say a torsion pair $(\mathcal{T},\mathcal{F})$
 is $\it{splitting}$ if every indecomposable $\Lambda$-module belongs to either $\mathcal{T}$ or $\mathcal{F}$.  Let $\Lambda$ have $\mathop{\text{domdim}}\Lambda\geq2$ and $(\mathcal{T}(T_{\mathcal{C}}),\mathcal{F}(T_{\mathcal{C}}))$ be the torsion pair induced by $T_{\mathcal{C}}$.
\begin{prop}$\emph{\cite[Proposition~2.2]{Z}}$.
\label{prop2}
The torsion pair $(\mathcal{T}(T_{\mathcal{C}}),\mathcal{F}(T_{\mathcal{C}}))$ is splitting if and only if $\pd_{\Lambda}X\leq1$ for every $X\in\mathcal{F}(T_{\mathcal{C}})$.
\end{prop}

 \subsection{Tilted Algebras}
 We now state the definition of a tilted algebra.
 \begin{mydef} Let $A$ be a hereditary algebra with $T$ a tilting $A$-module.  Then the algebra $\Lambda=\End_AT$ is called a $\emph{tilted algebra}$.
 \end{mydef}
The following lemma provides a useful criterion for a $\Lambda$-module to have projective or injective dimension at most $1$ in a tilted algebra.
\begin{lemma}$\emph{\cite[Lemma~3.1]{ZITO}}$ 
 \label{Homological Result}
 Let $\Lambda$ be a tilted algebra and $M$ a $\Lambda$-module.  
 \begin{enumerate}
 \item[$\emph{(1)}$] $\pd_{\Lambda}M\leq1$ if and only if $\emph{Hom}_{\Lambda}(\tau_{\Lambda}^{-1}\Omega_{\Lambda}^{-1}\Lambda,M)=0.$
 \item[$\emph{(2)}$] $\id_{\Lambda}M\leq1$ if and only if $\emph{Hom}_{\Lambda}(M,\tau_{\Lambda}\Omega_{\Lambda} D\Lambda)=0.$
 \end{enumerate}
 \end{lemma}
We note the above result was proved in the more general context of algebras of $\mathop{\text{gl.dim}}\Lambda=2$.
The next lemma gives homological dimension information on $\tau_{\Lambda}^{-1}\Omega_{\Lambda}^{-1}\Lambda$ and $\tau_{\Lambda}\Omega_{\Lambda} D\Lambda$.
\begin{lemma}
\label{something}
Let $\Lambda$ be a tilted algebra.
\begin{enumerate}
\item[$\emph{(1)}$] If $X\in\add(\tau_{\Lambda}^{-1}\Omega_{\Lambda}^{-1}\Lambda)$, then $\id_{\Lambda}X\leq1$.
\item[$\emph{(2)}$] If $X\in\add(\tau_{\Lambda}\Omega_{\Lambda} D\Lambda)$, then $\pd_{\Lambda}X\leq1$.
\end{enumerate}
\end{lemma}
\begin{proof}
We prove (1) with the proof of (2) similar.  If $\mathop{\text{gl.dim}}\Lambda\leq1$, then we are done.  If $\mathop{\text{gl.dim}}\Lambda=2$, Lemma $\ref{Homological Result}$ (1) implies $\pd_{\Lambda}X=2$ for every $X\in\add(\tau_{\Lambda}^{-1}\Omega_{\Lambda}^{-1}\Lambda)$.  It is well known $\pd_{\Lambda}X\leq1$ or $\id_{\Lambda}X\leq1$ for every $X\in\mathop{\text{ind}}\Lambda$ when $\Lambda$ is a tilted algebra.  Thus, we must have $\id_{\Lambda}X\leq1$ for every indecomposable 
$X\in\add(\tau_{\Lambda}^{-1}\Omega_{\Lambda}^{-1}\Lambda)$ and our result follows.
\end{proof}
The following result was shown in $\cite{Z}$ and gives a necessary and sufficient condition for an Auslander algebra to be tilted.  We note an algebra of global dimension less then or equal to $1$ is hereditary and thus tilted. 
\begin{theorem}$\emph{\cite[Theorem~2.3]{Z}}$
\label{MainZ}
Let $\Lambda$ be an Auslander algebra of $\mathop{\emph{gl.dim}}\Lambda=2$.  Then $\Lambda$ is tilted if and only if $\pd_{\Lambda}(\tau_{\Lambda}\Omega_{\Lambda}D\Lambda)\leq1$.
\end{theorem}
Our final result in this section shows the induced torsion pair, $(\mathcal{T}(T_{\mathcal{C}}),\mathcal{F}(T_{\mathcal{C}}))$, will always split when $\Lambda$ is an Auslander and tilted algebra. 
\begin{prop}
\label{splits}
Let $\Lambda$ be an Auslander and tilted algebra.  Let $(\mathcal{T}(T_{\mathcal{C}}),\mathcal{F}(T_{\mathcal{C}}))$ be the torsion pair induced by $T_{\mathcal{C}}$.  Then $(\mathcal{T}(T_{\mathcal{C}}),\mathcal{F}(T_{\mathcal{C}}))$ is splitting.
\end{prop}
\begin{proof}
Let $X\in\mathcal{F}(T_{\mathcal{C}})$.  We need to show $\pd_{\Lambda}X\leq1$.  Proposition $\ref{prop2}$ will then imply the torsion pair is splitting.  If $\pd_{\Lambda}X=2$, Lemma $\ref{Homological Result}$ (1) says $\text{Hom}_{\Lambda}(\tau_{\Lambda}^{-1}\Omega_{\Lambda}^{-1}\Lambda,X)\neq0.$
Since $X\in\mathcal{F}(T_{\mathcal{C}})=\text{Cogen}(\tau_{\Lambda}T_{\mathcal{C}})$, we have  $\text{Hom}_{\Lambda}(\tau_{\Lambda}^{-1}\Omega_{\Lambda}^{-1}\Lambda,\tau_{\Lambda}T_{\mathcal{C}})\neq0.$  Again, using Lemma $\ref{Homological Result}$ (1), $\pd_{\Lambda}(\tau_{\Lambda}T_{\mathcal{C}})=2$.  Since $T_{\mathcal{C}}$ is a cotilting module, Theorem $\ref{NRTZ Main}$ (3), says $\tau_{\Lambda}T_{\mathcal{C}}=\tau_{\Lambda}   (\bigoplus_i\Omega_{\Lambda}I_i)$ where $\Omega_{\Lambda}I_i$ is the syzygy of $I_i$ and the direct sum is taken over representatives of the isomorphism classes of all indecomposable injective non-projective $\Lambda$-modules $I_i$.  However, this a contradiction to Lemma $\ref{something}$ (2).  Thus, $\pd_{\Lambda}X\leq1$.  Since $X$ was arbitrary, we conclude the induced torsion pair $(\mathcal{T}(T_{\mathcal{C}}),\mathcal{F}(T_{\mathcal{C}}))$ is splitting.

\end{proof}

\subsection{$\mathcal{L}_{\Lambda}$, Projective Dimension, and Global Dimension}
Given $X,Y\in\mathop{\text{ind}}\Lambda$, we denote $X\leadsto Y$ in case there exists a chain of nonzero morphisms
\[ 
X=X_0\xrightarrow{f_1}X_1\xrightarrow{f_2}\cdots X_{t-1}\xrightarrow{f_t} X_t=Y
\]
with $t\geq0$, between indecomposable modules.  In this case we say $X$ is a predecessor of $Y$ and $Y$ is a predecessor of itself.  We now recall the definition of the left part of a module category.
\begin{mydef}
\label{def6}
Let $\Lambda$ be an algebra.  We denote by $\mathcal{L}_\Lambda$ the following subcategory of $\mathop{\text{ind}}\Lambda$:
\[
\mathcal{L}_{\Lambda}=\{Y\in\mathop{\text{ind}}\Lambda:\pd_{\Lambda}X\leq1~\text{for each}~X\leadsto Y\}.
\]
We call $\mathcal{L}_{\Lambda}$ the $\it{left~part}$ of the module category $\mathop{\text{mod}}\Lambda$.
\end{mydef}
It is easy to see that $\mathcal{L}_{\Lambda}$ is closed under predecessors and $\text{add}\mathcal{L}_{\Lambda}\subseteq\mathcal{P}^1(\Lambda)$.
The next result gives necessary and sufficient conditions for an indecomposable module $Y$ to be in $\mathcal{L}_{\Lambda}$.
\begin{theorem}$\emph{\cite[Theorem~1.1]{ALR}}$
\label{left}
Let $\Lambda$ be an algebra with $Y\in\mathop{\emph{ind}}\Lambda$. Then $Y\in\mathcal{L}_{\Lambda}$ if and only if, for every $X\in\mathop{\emph{ind}}\Lambda$ with projective dimension at least two, we have $\emph{Hom}_{\Lambda}(X,Y)=0$.
\end{theorem}
The following result on projective dimensions is needed.
 \begin{lemma}$\emph{\cite[Appendix,~Proposition~4.7]{ASS}}$. 
 \label{Standard Projective Restrictions}
 Let $\Lambda$ be an algebra and suppose that $0\rightarrow L\rightarrow M\rightarrow N\rightarrow 0$ is a short exact sequence in $\mathop{\emph{mod}}\Lambda$.
\[
\pd_{\Lambda}N\leq \emph{max}(\pd_{\Lambda}M,1+\pd_{\Lambda}L), \emph{and equality holds if} \pd_{\Lambda}M\neq\pd_{\Lambda}L.
\]
 \end{lemma}
In computing the global dimension of an algebra, the following theorem due to Auslander is very useful.
\begin{theorem}$\emph{\cite{AU1}}$
\label{Global}
If $\Lambda$ is an algebra, then 
\[
\mathop{\emph{gl.dim}}\Lambda=1+\emph{max}\{\pd_{\Lambda}(\emph{rad}e\Lambda); e\in\Lambda~is~a~primitive~idempotent\}.
\]

\end{theorem}

\section{Main Result}
 We begin with two straightforward propositions.
 \begin{prop}
 \label{easy1}
 Suppose $\Lambda$ is a $1$-Auslander-Gorenstein algebra.  Denote by $\mathcal{C}_{\Lambda}$ the 
the full subcategory consisting of all modules generated and cogenerated by the direct sum of all indecomposable projective-injective $\Lambda$-modules.  Let $T_{\mathcal{C}}$ denote the unique tilting-cotilting module in $\mathcal{C}_{\Lambda}$.  Then  $\mathcal{P}^1(\Lambda)\subseteq\emph{Cogen}T_{\mathcal{C}}$. 
 \end{prop}
 \begin{proof}
 Let $X\in\mathcal{P}^1(\Lambda)$.  If $X$ is projective, then $X\in\text{Cogen}T_{\mathcal{C}}$ by the definition of a tilting module.  If $\pd_{\Lambda}X=1$, then Lemma $\ref{small}$ gives $\text{Ext}_{\Lambda}^1(X,T_{\mathcal{C}})=0$.  Since $T_{\mathcal{C}}$ is a cotilting module, we must have $X\in\text{Cogen}T_{\mathcal{C}}$. 
 \end{proof}
 
 \begin{prop}
 \label{easy2}
Suppose $\Lambda$ is a $1$-Auslander-Gorenstein algebra.  Denote by $\mathcal{C}_{\Lambda}$ the 
the full subcategory consisting of all modules generated and cogenerated by the direct sum of all indecomposable projective-injective $\Lambda$-modules.  Let $T_{\mathcal{C}}$ denote the unique tilting-cotilting module in $\mathcal{C}_{\Lambda}$.  Then $\Lambda$ is an Auslander algebra if and only if $\mathcal{P}^1(\Lambda)=\emph{Cogen}T_{\mathcal{C}}$.
\end{prop}
\begin{proof}
Assume $\Lambda$ is an Auslander algebra.  We know from Proposition $\ref{easy1}$ that $\mathcal{P}^1(\Lambda)\subseteq\text{Cogen}T_{\mathcal{C}}$.  Now, assume $X\in\text{Cogen}T_{\mathcal{C}}$ but $X\not\in\mathcal{P}^1(\Lambda)$.  Since $\Lambda$ is an Auslander algebra, we must have $\pd_{\Lambda}X=2$.  Since $X\in\text{Cogen}T_{\mathcal{C}}$, we have a short exact sequence
\[
0\rightarrow X\rightarrow T_{\mathcal{C}}'\rightarrow Y\rightarrow 0 
\]
where $T_{\mathcal{C}}'\in\add T_{\mathcal{C}}$.  Since $\pd_{\Lambda}T_{\mathcal{C}}\leq1$, Lemma $\ref{Standard Projective Restrictions}$ says $\pd_{\Lambda}Y=3$ and we have a contradiction.  Thus,  
$\mathcal{P}^1(\Lambda)=\text{Cogen}T_{\mathcal{C}}$.
\par
Next, assume $\mathcal{P}^1(\Lambda)=\text{Cogen}T_{\mathcal{C}}$.  Let $P$ be any indecomposable projective module.  Since $T_{\mathcal{C}}$ is a tilting module, we know $P\in\text{Cogen}T_{\mathcal{C}}$.  Thus, the radical of $P$ belongs to $\text{Cogen}T_{\mathcal{C}}$ and must have projective dimension less than or equal to one.  Theorem $\ref{Global}$ implies the global dimension of $\Lambda$ is finite.  Proposition $\ref{general}$ gives $\Lambda$ is an Auslander algebra.
\end{proof}
We are now ready to prove our main result.
\begin{theorem}
Suppose $\Lambda$ is a $1$-Auslander-Gorenstein algebra.  Denote by $\mathcal{C}_{\Lambda}$  
the full subcategory consisting of all modules generated and cogenerated by the direct sum of all indecomposable projective-injective $\Lambda$-modules.  Let $T_{\mathcal{C}}$ denote the unique tilting-cotilting module in $\mathcal{C}_{\Lambda}$.  Then $\Lambda$ is a tilted algebra if and only if $\emph{add}\mathcal{L}_{\Lambda}=\emph{Cogen}T_{\mathcal{C}}$.
\end{theorem}

\begin{proof}
Assume $\text{add}\mathcal{L}_{\Lambda}=\text{Cogen}T_{\mathcal{C}}$.  First, let's show $\Lambda$ is an Auslander algebra.  We know $\mathcal{P}^1(\Lambda)\subseteq\text{Cogen}T_{\mathcal{C}}$ from Proposition $\ref{easy1}$ and $\text{add}\mathcal{L}_{\Lambda}\subseteq\mathcal{P}^1(\Lambda)$ from Definition 
$\ref{def6}$.  Thus, $\mathcal{P}^1(\Lambda)=\text{Cogen}T_{\mathcal{C}}$ and Proposition $\ref{easy2}$ gives $\Lambda$ is an Auslander algebra.
\par
By our original assumption, we certainly have $T_{\mathcal{C}}\in\text{add}\mathcal{L}_{\Lambda}$.  Thus, $\tau_{\Lambda}(T_{\mathcal{C}})\in\text{add}\mathcal{L}_{\Lambda}$ since $\mathcal{L}_{\Lambda}$ is closed under predecessors.  We know from Theorem $\ref{NRTZ Main}$ (3) that $\tau_{\Lambda}(T_{\mathcal{C}})=\tau_{\Lambda}(\bigoplus_i\Omega_{\Lambda}I_i)$ where $\Omega_{\Lambda}I_i$ is the syzygy of $I_i$ and the direct sum is taken over representatives of the isomorphism classes of all indecomposable injective non-projective $\Lambda$-modules $I_i$.  Theorem $\ref{MainZ}$ implies $\Lambda$ is tilted.
\par
Now, assume $\Lambda$ is a tilted algebra.  We know $\mathcal{P}^1(\Lambda)=\text{Cogen}T_{\mathcal{C}}$ from Proposition $\ref{easy2}$ and $\text{add}\mathcal{L}_{\Lambda}\subseteq\mathcal{P}^1(\Lambda)$ from Definition $\ref{def6}$.  This gives $\text{add}\mathcal{L}_{\Lambda}\subseteq\text{Cogen}T_{\mathcal{C}}$.  Next, assume $X$ is an indecomposable module such that $X\in\text{Cogen}T_{\mathcal{C}}$ but $X\not\in\mathcal{L}_{\Lambda}$.  By Theorem $\ref{left}$, there exists an indecomposable module $Y$ with $\pd_{\Lambda}Y=2$ such that $\text{Hom}_{\Lambda}(Y,X)\neq0$.  Since $X\in\text{Cogen}T_{\mathcal{C}}$, we have $\text{Hom}_{\Lambda}(Y,T_{\mathcal{C}})\neq0$.
\par
Consider the induced torsion pair $(\mathcal{T}(T_{\mathcal{C}}),\mathcal{F}(T_{\mathcal{C}}))$.  Since $\Lambda$ is tilted, Proposition $\ref{splits}$ implies the torsion pair is splitting.  Thus, $Y\in\mathcal{T}(T_{\mathcal{C}})$ or $Y\in\mathcal{F}(T_{\mathcal{C}})$.  If $Y\in\mathcal{F}(T_{\mathcal{C}})$, then Theorem $\ref{NRTZ Main}$ (3) and Lemma $\ref{Homological Result}$ (2) imply $\id_{\Lambda}Y=2$.
Since $\pd_{\Lambda}Y=2$, this contradicts the well known fact that $\pd_{\Lambda}M\leq1$ or $\id_{\Lambda}M\leq1$ for every $M\in\mathop{\text{ind}}\Lambda$ when $\Lambda$ is a tilted algebra.
We conclude $\id_{\Lambda}Y\leq1$ and $Y\in\mathcal{T}(T_{\mathcal{C}})$.
\par
Next, we wish to show $\tau_{\Lambda}Y\in\mathcal{T}(T_{\mathcal{C}})$.  Since $Y$ is an indecomposable non-projective module, we know $\tau_{\Lambda}Y$ is an indecomposable non-zero module.  Since $(\mathcal{T}(T_{\mathcal{C}}),\mathcal{F}(T_{\mathcal{C}}))$ is splitting, if $\tau_{\Lambda}Y\not\in\mathcal{T}(T_{\mathcal{C}})$, then $\tau_{\Lambda}Y\in\mathcal{F}(T_{\mathcal{C}})$.
Proposition $\ref{Ext}$ then implies $Y$ is Ext-projective in $\mathcal{T}(T_{\mathcal{C}})$ and Theorem $\ref{Tilting}$ shows $Y\in\add T_{\mathcal{C}}$.  With $\pd_{\Lambda}Y=2$ and $T_{\mathcal{C}}$ a tilting module, we have a contradiction.
\par
Since $\tau_{\Lambda}Y\in\mathcal{T}(T_{\mathcal{C}})$, we known $\Ext^1_{\Lambda}(T_{\mathcal{C}},\tau_{\Lambda}Y)=0$.  We also know, by the Auslander-Reiten formulas and the fact $T_{\mathcal{C}}$ is a tilting module, that $D\Hom_{\Lambda}(\tau_{\Lambda}Y,\tau_{\Lambda}(T_{\mathcal{C}}))=0$.  Now, Theorem $\ref{NRTZ Main}$ (3) shows $\tau_{\Lambda}(T_{\mathcal{C}})=\tau_{\Lambda}(\bigoplus_i\Omega_{\Lambda}I_i)$.  Lemma $\ref{Homological Result}$ (2) gives $\id_{\Lambda}(\tau_{\Lambda}Y)\leq1$.  Again, using the Auslander-Reiten formulas, $\Ext^1_{\Lambda}(T_{\mathcal{C}},\tau_{\Lambda}Y)=0$ implies $D\Hom_{\Lambda}(\tau_{\Lambda}^{-1}(\tau_{\Lambda}Y),T_{\mathcal{C}})\cong\ D\Hom_{\Lambda}(Y,T_{\mathcal{C}})=0$.  This is a contradiction and we conclude $X\in\mathcal{L}_{\Lambda}$.  Thus, $\text{add}\mathcal{L}_{\Lambda}=\text{Cogen}T_{\mathcal{C}}$.
\end{proof}

\noindent Mathematics Faculty, University of Connecticut-Waterbury, Waterbury, CT 06702, USA
\it{E-mail address}: \bf{stephen.zito@uconn.edu}

\end{document}